\documentclass[10pt]{amsart}
\usepackage{framed}
\usepackage{color}

\usepackage[T1]{fontenc}
\usepackage{ae,aecompl}
\newcommand{\defeq}{\stackrel{{\text{def}}}{=}}
\newcommand{\pp}[2]{\frac{\partial #1}{\partial #2}}
\newcommand{\pppp}[3]{\frac{\partial^2 #1}{\partial #2 \partial #3}}

\newtheorem{proposition}{Proposition}
\newtheorem{lemma}{Lemma}
\newtheorem{remark}{Remark}
\newtheorem{theorem}{Theorem}
\usepackage{comment}
\usepackage{wrapfig,subfig,tikz}
\usepackage{amsrefs}
\usepackage{url}

\begin{document}
\date\today
\title[Characteristic Initial Value Problem]{Characteristic Initial Value Problem for Spherically Symmetric Barotropic Flow}
\author{Andr\'e Lisibach}
\address{Department of Mathematics\\Princeton University}
\email{lisibach@princeton.edu}
\begin{abstract}
We study the equations of motion for a barotropic fluid in spherical symmetric flow. Making use of the Riemann invariants we consider the characteristic form of these equations. In a first part, we show that the resulting constraint equations along characteristics can be solved globally away from the center of symmetry. In a second part, given data on two intersecting characteristics, we show existence and uniqueness of a smooth solution in a neighborhood in the future of these characteristics.
\end{abstract}
\maketitle

\section{Introduction}
%\begin{comment}
The equations of motion describing a compressible inviscid fluid are of hyperbolic type. For such equations, along a characteristic hypersurface not all of the unknown functions can be prescribed freely. The ones that can will be denoted in the following as \textit{free data}. When restricting the equations of motion to a characteristic hypersurface they become the so called \textit{constraint equations}. Given free data these are equations for the remaining unknowns which will be called \textit{derived data}. This situation is in contrast to the Cauchy problem where all of the unknown functions can be prescribed at $t=0$. The system of constraint equations is nonlinear, therefore a solution does not exist in general. We study the constraint equations for the Euler equations in the case of a barotropic fluid, i.e.~under the assumption that $p=f(\rho)$, where $p$, $\rho$ are the pressure and density of the fluid respectively. In addition we assume that the flow is spherically symmetric, hence the problem reduces to one in the $t$-$r$-plane, where $t$, $r$ denote the time and radial coordinate respectively. We use the Riemann invariants which lead to a natural formulation of the equations of motion along characteristics. The resulting constraint equations form a two by two system of nonlinear ordinary differential equations. In the first part we show existence and uniqueness of a solution of this system globally away from the center of symmetry $r=0$.

Once the constraint equations are solved and therefore characteristic data has been established, a natural follow up question is whether one can find a solution of the equations of motion in a neighborhood of two intersecting characteristic hypersurfaces in the acoustical future of the intersection. This is the content of the second part of the present work.

It is important to note that the solution thus obtained corresponds to a solution in the $t$-$r$-plane only where the jacobian of the transformation from the characteristic coordinates to the $t$-$r$-plane does not vanish (see page \pageref{tr_solution}). Such points of vanishing jacobian represent points in the singular part of the boundary of the maximal development (see for example chapter 2 of \cite{Christodoulou_2007}). Therefore, such points and their range of influence have to be excluded from the solution. Furthermore, in view of obtaining a physically acceptable solution, we note that an even further restriction might apply once a shock solution beyond a point of blowup has been established, the shock lying in the past of the boundary of the maximal development (see \cite{Christodoulou_Lisibach}).

\begin{comment}
We remark that in the case of the Einstein equations, all constraint equations are linear ordinary differential equations along the generators of the characteristic hypersurfaces and all except one of these equations are of first order. The one not being of first order is that for the conformal factor of the metric of the sections of the characteristic hypersurfaces. This satisfies a second order ordinary differential equation. This equation is also linear so formally a solution exists for all positive values of the affine parameter. However only an everywhere positive solution is physically acceptable, and the rest of the construction of the derived data depends on this positivity. There is an open set of free data for which positivity fails. This is analyzed in pages 62-78 of \cite{Christodoulou}.

We remark that in the case of the Einstein equations the constraint equations along characteristics for generic free data can not be solved. Given data on two intersecting characteristic hypersurfaces the constraint equations which in this case become a system of transport equations can be solved only locally near the intersection (see \cite{Rendall}). Luk has shown however that once a solution of the constraint equations along the intersecting characteristic hypersurfaces is given, the Einstein equations can be solved locally in the future of the intersection in a neighborhood of the two intersecting characteristic hypersurfaces \cite{Luk}.
\end{comment}

The present work can be viewed as a first step towards understanding the characterisitic initial value problem for the Euler equations without any symmetry assumptions.

%\end{comment}

\section{Equations of Motion, Characteristic System}
We review the basic equations needed for the study of a barotropic fluid in spherical symmetry.

\subsection{Equations of Motion in Spherical Symmetry}
We denote by $w$, $\rho$, $p$ the fluid velocity, the density and the pressure, respectively. We assume a barotropic equation of state, i.e.~$p=f(\rho)$, and we assume $f\in C^\infty$, $dp/d\rho,d^2p/d\rho^2>0$. The adiabatic condition decouples and we are left with
\begin{align}
  \label{eq:1}
  \partial_t\rho+\partial_r(\rho w)&=-\frac{2\rho w}{r},\\
  \label{eq:2}
  \partial_tw+w\partial_rw&=-\frac{\eta^2}{\rho}\partial_r\rho,
\end{align}
where we denote by $\eta$ the sound speed, i.e.~$\eta^2=dp/d\rho$. We assume $\rho>0$, i.e.~we exclude vacuum.

\subsection{Riemann Invariants, Characteristic System}
Let (see \cite{Riemann})
\refstepcounter{equation}\label{eq:5ab}
\begin{align}
  \label{eq:5}
  \alpha\defeq\int^{\rho}\frac{\eta(\rho')}{\rho'}d\rho'+w,\qquad\beta\defeq\int^{\rho}\frac{\eta(\rho')}{\rho'}d\rho'-w\tag{\theequation a,b}
\end{align}
and
\refstepcounter{equation}\label{eq:6ab}
\begin{align}
  c_\pm\defeq w\pm\eta,\qquad L_\pm\defeq\partial_t+c_\pm\partial_r.\tag{\theequation a,b}\label{eq:6}
\end{align}
We have
\begin{align}
  \label{eq:7}
  L_+\alpha=L_-\beta=-\frac{\eta(\alpha,\beta)}{r}(\alpha-\beta)\defeq F(\alpha,\beta,r).
\end{align}
Introducing the coordinates $u$, $v$ such that $u$ is constant along integral curves of $L_+$ and $v$ is constant along integral curves of $L_-$, \eqref{eq:7} becomes
\refstepcounter{equation}\label{eq:8ab}
\begin{align}
  \label{eq:8}
  \pp{\alpha}{v}=\pp{t}{v}F(\alpha,\beta,r),\qquad\pp{\beta}{u}=\pp{t}{u}F(\alpha,\beta,r)\tag{\theequation a,b}.
\end{align}
$t$ and $r$ satisfy the Hodograph system
\refstepcounter{equation}\label{eq:9ab}
\begin{align}
  \label{eq:9}
  \pp{r}{v}=\pp{t}{v}c_+(\alpha,\beta),\qquad \pp{r}{u}=\pp{t}{u}c_-(\alpha,\beta).\tag{\theequation a,b}
\end{align}
In the following we refer to \eqref{eq:8}, \eqref{eq:9} as the characteristic system of equations (see \cite{Courant_Friedrichs}).

From (\ref{eq:5}) we have
\begin{align}
  \label{eq:19}
  \frac{\partial(\alpha,\beta)}{\partial(\rho,w)}&=\left(
    \begin{array}{cc}
      \eta/\rho & 1\\
      \eta/ \rho & -1
    \end{array}\right).
\intertext{Therefore,}
  \label{eq:20}
  \frac{\partial(\rho,w)}{\partial(\alpha,\beta)}&=\left(
    \begin{array}{cc}
      \rho/2\eta & \rho/2\eta\\
      1/2 & -1/2
    \end{array}\right).
\end{align}
Let now
\begin{align}
  \label{eq:23}
  \chi\defeq\alpha-\beta,\qquad\chi^\dagger\defeq \alpha+\beta.
\end{align}
We have
\begin{align}
  \label{eq:24}
  \pp{\alpha}{\chi^\dagger}=\pp{\beta}{\chi^\dagger}=\frac{1}{2}.
\end{align}
Now,
\begin{align}
  \label{eq:22}
  \pp{\eta}{\chi^\dagger}&=\frac{d\eta}{d\rho}\left\{\pp{\rho}{\alpha}\pp{\alpha}{\chi^\dagger}+\pp{\rho}{\beta}\pp{\beta}{\chi^\dagger}\right\}\notag\\
&=\frac{\rho}{4\eta^2}\frac{d^2p}{d\rho^2}>0.
\end{align}
Similarly we find $\partial\eta/\partial\chi=0$. Therefore, $\eta=\eta(\chi^\dagger)$ and $d\eta/d\chi^\dagger>0$.

\section{Characteristic Initial Data}
We look at a point $(t_0,r_0)$ in the $t$-$r$-plane and denote the outgoing and incoming characteristic originating from this point by $C^+$ and $C^-$ respectively. We put the origin of the $u$-$v$-coordinates at $(t_0,r_0)$ and we set $t_0=0$.

In view of (\ref{eq:8ab}a), (\ref{eq:9ab}a) the free data on $C^+$ consists of $\beta^+(v)=\beta(0,v)$, $t^+(v)= t(0,v)$ for an increasing function $t^+$. We fix the coordinate $v$ along $C^+$ by setting $t^+(v)=v$. Then (\ref{eq:8ab}a), (\ref{eq:9ab}a) constitute the following system of nonlinear ode for the derived data $\alpha^+(v)$, $r^+(v)$ on $C^+$:
\begin{subequations}
\label{eq:18}
\begin{align}
  \label{eq:12}
  \frac{d\alpha}{dv}&=-\frac{\eta(\alpha,\beta)}{r}(\alpha-\beta),\\
  \label{eq:13}
  \frac{dr}{dv}&=\tfrac{1}{2}(\alpha-\beta)+\eta(\alpha,\beta),
\end{align}
\end{subequations}
where we omitted the superscript $+$ on $\alpha$, $\beta$ and $r$.

In view of (\ref{eq:8ab}b), (\ref{eq:9ab}b) the free data on $C^-$ consists of $\alpha^-(u)=\alpha(u,0)$, $t^-(u)=t(u,0)$ for an increasing function $t^-$. We fix the coordinate $u$ along $C^-$ by setting $t^-(u)=u$. Then (\ref{eq:8ab}b), (\ref{eq:9ab}b) constitute the following system of nonlinear ode for the derived data $\beta^-(u)$, $r^-(u)$ on $C^-$:
\begin{subequations}
\label{eq:17}
  \begin{align}
    \label{eq:15}
    \frac{d\beta}{du}&=-\frac{\eta(\alpha,\beta)}{r}(\alpha-\beta),\\
      \label{eq:16}
    \frac{dr}{du}&=\tfrac{1}{2}(\alpha-\beta)-\eta(\alpha,\beta),
  \end{align}
\end{subequations}
where we omitted the superscript $-$ on $\alpha$, $\beta$ and $r$. \eqref{eq:18}, \eqref{eq:17} are the constraint equations along $C^+$, $C^-$, respectively. The following lemma shows that there exists derived smooth data on $C^+$ and on $C^-$ as long as $C^-$ stays away from $r=0$.

\begin{lemma}
  Let $\alpha^-,\beta^+\in C^\infty(\mathbb{R}^+\cup\{0\})$ and $r_0>0$. Then
\begin{itemize}
\item [i)] the system \eqref{eq:18} with $\beta=\beta^+$, $\alpha(0)=\alpha^-(0)$ and $r(0)=r_0$ has a solution for $v\in\mathbb{R}^+\cup\{0\}$,
\item [ii)] for any $\varepsilon>0$ with $\varepsilon<r_0$ the system \eqref{eq:17} with $\alpha=\alpha^-$, $\beta(0)=\beta^+(0)$ and $r(0)=r_0$ has a solution for $u\in [0,\overline{u})$, where $\overline{u}=\sup \Big\{u'\in\mathbb{R}^+\cup\{0\}: \forall u''\in[0,u'] : r(u'')>\varepsilon\Big\}$.
\end{itemize}
\end{lemma}

\begin{proof}
Part i). The only possibilities for the system \eqref{eq:18} to blow up are
\begin{align}
  \label{eq:29}
  r\rightarrow 0,\qquad r\rightarrow\infty,\qquad |\alpha|\rightarrow\infty.
\end{align}
Using $\chi=\alpha-\beta$, the system \eqref{eq:18} becomes (we omit the argument of $\eta$)
\begin{subequations}
\begin{align}
  \label{eq:25}
  \frac{d\chi}{dv}&=-\frac{\eta}{r}\chi-\frac{d\beta}{dv},\\
  \label{eq:27}
  \frac{dr}{dv}&=\eta+\tfrac{1}{2}\chi.
\end{align}
\end{subequations}
From \eqref{eq:25} we obtain
\begin{align}
  \label{eq:28}
  \chi(v)=\chi(0)e^{-\int_0^v\left(\frac{\eta}{r}\right)(v')dv'}-\int_0^ve^{-\int_{v'}^v\left(\frac{\eta}{r}\right)(v'')dv''}\frac{d\beta}{dv}(v')dv'.
\end{align}
Therefore,
\begin{align}
  \label{eq:30}
  |\chi(v)|\leq|\chi(0)|+\int_0^v\left|\frac{d\beta}{dv}(v')\right|dv',
\end{align}
which implies that the absolute value of $\chi$ is bounded. The bound on $\chi$ implies that also $\alpha$ is bounded in absolute value, i.e.~$|\alpha|\leq C$. Hence also $\chi^\dagger$ is bounded in absolute value, i.e.~$|\chi^\dagger|\leq C$. From \eqref{eq:20}, \eqref{eq:23} we have
  \begin{align}
    \label{eq:199}
    \frac{d\log\rho}{d\chi^\dagger}=\frac{1}{4\eta}.
  \end{align}
In view of $\eta>0$, we deduce that $\rho\leq C$. This in turn implies that $\eta$ is bounded. From \eqref{eq:27} we have
\begin{align}
  \label{eq:31}
  r(v)=r_0+\int_0^v(\eta+\tfrac{1}{2}\chi)(v')dv'.
\end{align}
Therefore, $r$ is bounded from above.

The only possibility for blowup left to study is $r\rightarrow 0$. Let us assume
\begin{align}
  \label{eq:47}
  r\rightarrow 0\quad\textrm{as}\quad v\rightarrow v^\ast.
\end{align}
Let $0<v_1<v<v^\ast$. Integrating \eqref{eq:25} on $[v_1,v]$ yields
\begin{align}
  \label{eq:32}
  \chi(v)\geq \chi(v_1)e^{-\int_{v_1}^v\left(\frac{\eta}{r}\right)(v')dv'}-C(v-v_1).
\end{align}
Let
\begin{align}
  \label{eq:33}
  \underline{\eta}\defeq \inf_{[0,v^\ast]}\eta.
\end{align}
We have
\begin{align}
  \label{eq:34}
  \eta(v)\geq\underline{\eta}>0.
\end{align}
Let $0<\tilde{C}<\underline{\eta}$. Using the lower bounds for $\chi$ and $\eta$ as given by \eqref{eq:32} and \eqref{eq:34}, respectively, in \eqref{eq:27} and choosing $v_1\in(0,v^\ast)$ such that
\begin{align}
  \label{eq:35}
  v_1\geq v^\ast-\frac{2}{C}(\underline{\eta}-\tilde{C}),
\end{align}
where $C$ is the constant appearing in \eqref{eq:32}, we obtain
\begin{align}
  \label{eq:36}
  \frac{dr}{dv}(v)\geq \tilde{C}+\tfrac{1}{2}\chi(v_1)e^{-\int_{v_1}^v\left(\frac{\eta}{r}\right)(v')dv'},
\end{align}
Defining
\begin{align}
  \label{eq:37}
  \tilde{\chi}(v)\defeq \chi(v_1)e^{-\int_{v_1}^v\left(\frac{\eta}{r}\right)(v')dv'},
\end{align}
\eqref{eq:36} becomes
\begin{align}
  \label{eq:38}
  \frac{dr}{dv}(v)\geq \tilde{C}+\tfrac{1}{2}\tilde{\chi}(v).
\end{align}
In the case $\chi(v_1)\geq 0$ we have $\tilde{\chi}\geq 0$ and $(dr/dv)(v)\geq \tilde{C}>0$ for $v\in[v_1,v^\ast]$, which contradicts \eqref{eq:47}. We consider the case $\chi(v_1)< 0$. From \eqref{eq:37} we see that in $[v_1,v^\ast]$, $\tilde{\chi}$ is monotonically increasing and $\tilde{\chi}<0$. We define $\tilde{\phi}\defeq-\tilde{\chi}$. Then, on $[v_1,v^\ast]$, $\tilde{\phi}>0$, $\tilde{\phi}$ is monotonically decreasing and
\begin{align}
  \label{eq:39}
  \frac{dr}{dv}(v)\geq \tilde{C}-\tfrac{1}{2}\tilde{\phi}(v).
\end{align}

We look at the subcase $2\tilde{C}\geq \tilde{\phi}(v_1)$. In this subcase we have $2\tilde{C}\geq \tilde{\phi}(v)$ for $v\in[v_1,v^\ast]$, which implies $(dr/dv)(v)\geq 0$ in $[v_1,v^\ast]$. Together with $r(v_1)>0$ this contradicts \eqref{eq:47}.

We look at the subcase $2\tilde{C}<\tilde{\phi}(v_1)$. Since $\tilde{\phi}$ is a monotonically decreasing function, it either drops not below $2\tilde{C}$ or it drops below $2\tilde{C}$ on $[v_1,v^\ast]$. In the latter situation there is a $v_0\in[v_1,v^\ast]$ such that $(dr/dv)(v)\geq 0$ for $v\in[v_0,v^\ast]$ and we get a contradiction in the same way as in the subcase studied above.

Therefore, we are left to study the situation in which $\tilde{\phi}$ does not drop below $2\tilde{C}$ in $[v_1,v^\ast]$. Since $\tilde{\phi}$ is a decreasing function which is bounded from below by $2\tilde{C}$, it tends to a limit as $v\rightarrow v^\ast$. We have
\begin{align}
  \label{eq:40}
  \lim_{v\rightarrow v^\ast}\tilde{\phi}(v)\geq 2\tilde{C}.
\end{align}
Now,
\begin{align}
  \label{eq:41}
  -\frac{dr}{dv}(v)&\leq \tfrac{1}{2}\tilde{\phi}(v)-\tilde{C}\notag\\
&\leq\tfrac{1}{2}\tilde{\phi}(v_1)-\tilde{C}.
\end{align}
Integrating this on $[v,v^\ast]$ for $v\in[v_1,v^\ast]$ and taking into account the assumption \eqref{eq:47}, we obtain
\begin{align}
  \label{eq:42}
  r(v)\leq \left(\tfrac{1}{2}\tilde{\phi}(v_1)-\tilde{C}\right)(v^\ast-v).
\end{align}
Using this together with the lower bound on $\eta$ we deduce
\begin{align}
  \label{eq:43}
  \int_{v_1}^v\left(\frac{\eta}{r}\right)(v')dv'&\geq \frac{\underline{\eta}}{\frac{1}{2}\tilde{\phi}(v_1)-\tilde{C}}\int_{v_1}^v\frac{dv'}{v^\ast-v'}\notag\\
&=\frac{\underline{\eta}}{\frac{1}{2}\tilde{\phi}(v_1)-\tilde{C}}\log\left(\frac{v^\ast-v_1}{v^\ast-v}\right).
\end{align}
Recalling the definition of $\tilde{\phi}$ we get
\begin{align}
  \label{eq:44}
  \tilde{\phi}(v)&=\tilde{\phi}(v_1)e^{-\int_{v_1}^v\left(\frac{\eta}{r}\right)(v')dv'}\notag\\
&\leq\tilde{\phi}(v_1)\left(\frac{v^\ast-v}{v^\ast-v_1}\right)^{\frac{\underline{\eta}}{\frac{1}{2}\tilde{\phi}(v_1)-\tilde{C}}}.
\end{align}
The right hand side tends to $0$ as $v\rightarrow v^\ast$ giving us a contradiction to \eqref{eq:40}.

Part ii). Assuming $r>\varepsilon$ the only possibilities for the system \eqref{eq:17} to blow up are
\begin{align}
  \label{eq:45}
  r\rightarrow \infty,\qquad |\beta|\rightarrow \infty.
\end{align}
Using $\chi=\alpha-\beta$, the system \eqref{eq:17} becomes (we omit the argument of $\eta$)
\begin{subequations}
\begin{align}
  \label{eq:46}
  \frac{d\chi}{du}&=\frac{\eta}{r}\chi+\frac{d\alpha}{du},\\
  \label{eq:48}
  \frac{dr}{du}&=\tfrac{1}{2}\chi-\eta.
\end{align}
\end{subequations}
Let $0<u_1<u<u^\ast$. Integrating \eqref{eq:46} on $[u_1,u]$ yields
\begin{align}
  \label{eq:49}
  \chi(u)=e^{\int_{u_1}^u\left(\frac{\eta}{r}\right)(u')du'}(\chi(u_1)+F_1(u)),
\end{align}
where
\begin{align}
  \label{eq:54}
  F_1(u)\defeq \int_{u_1}^ue^{-\int_{u_1}^{u'}\left(\frac{\eta}{r}\right)(u'')du''}\frac{d\alpha}{du}(u')du'.
\end{align}
We have
\begin{align}
  \label{eq:55}
  |F_1(u)|\leq C(u-u_1)
\end{align}
and
\begin{align}
  \label{eq:57}
  |\chi(u)|\leq Ce^{\int_{u_1}^u\left(\frac{\eta}{r}\right)(u')du'}.
\end{align}

Let us assume
\begin{align}
  \label{eq:52}
  |\chi|\rightarrow\infty \quad\textrm{as}\quad u\rightarrow u^\ast.
\end{align}
In view of the bound \eqref{eq:55} we see that there exists $u_1\in[0,u^\ast)$ such that for $u\in[u_1,u^\ast)$ we have either $\chi(u_1)+F_1(u)>0$ or $\chi(u_1)+F_1(u)<0$. Using this in \eqref{eq:49} we obtain that either $\chi(u)>0$ for $u\in[u_1,u^\ast)$ and $\chi\rightarrow \infty$ as $u\rightarrow u^\ast$ or $\chi(u)<0$ for $u\in[u_1,u^\ast)$ and $\chi\rightarrow -\infty$ as $u\rightarrow u^\ast$ respectively.

In the first case, $\chi\rightarrow\infty$ as $u\rightarrow u^\ast$ implies $\beta\rightarrow-\infty$ as $u\rightarrow u^\ast$ which in turn implies $\chi^\dagger\rightarrow-\infty$ as $u\rightarrow u^\ast$. In view of \eqref{eq:22} we then have $\lim_{u\rightarrow u^\ast}\eta(u)\leq C$. Using this together with $1/r\leq C$ we obtain that the right hand side of \eqref{eq:57} with $u=u^\ast$ is bounded. This contradicts \eqref{eq:52}.

In the second case, $\chi\rightarrow -\infty$ as $u\rightarrow u^\ast$ implies, through \eqref{eq:57} and $1/r\leq C$, that
\begin{align}
  \label{eq:58}
  \int_{u_1}^u\eta(u')du'\rightarrow \infty \quad\textrm{as}\quad u\rightarrow u^\ast.
\end{align}
Integrating \eqref{eq:48} gives
\begin{align}
  \label{eq:59}
  r(u)=r(u_1)+\int_{u_1}^u\left(\tfrac{1}{2}\chi(u')-\eta(u')\right)du'.
\end{align}
Using \eqref{eq:58} together with $\chi\rightarrow-\infty$, we see that $r(u)\rightarrow 0$ as $u\rightarrow u^\ast$, contradicting our assumption $r>0$.

Therefore $\chi$ is bounded and hence so is $\beta$, i.e.~$|\beta|\leq C$. The bounds on $\alpha$ and $\beta$ imply that also $\eta$ is bounded (see \eqref{eq:199}). In view of \eqref{eq:48} we then obtain an upper bound for $r(u)$.

\end{proof}

The above lemma shows that there is no blowup of $\alpha$, $r$ along $C^+$ and of $\beta$, $r$ along $C^-$ as long as $C^-$ does not hit the center of symmetry $r=0$, thus establishing a continuously differentiable solution of the constraint equations. In the following we show that there is also no blowup for higher order derivatives of $\alpha$, $\beta$, $t$ and $r$.

We define
\refstepcounter{equation}\label{eq:61ab}
\begin{align}
  \label{eq:61}
  \mu\defeq\pp{t}{u},\qquad \nu\defeq\pp{t}{v},\tag{\theequation a,b}
\end{align}
and
\refstepcounter{equation}\label{eq:62ab}
\begin{align}
  \label{eq:62}
  \gamma\defeq\pp{\alpha}{u},\qquad\delta\defeq\pp{\beta}{v}.\tag{\theequation a,b}
\end{align}
In the following we denote by $f_i:i=1,\ldots,10$ given continuously differentiable functions. Taking the derivative of (\ref{eq:8ab}a) with respect to $u$ and making use of (\ref{eq:8ab}b), (\ref{eq:9ab}b) we obtain the following equation along $C^+$
\begin{align}
  \label{eq:63}
  \frac{d\gamma}{dv}=f_1\frac{d\mu}{dv}+f_2\gamma+f_3\mu.
\end{align}
Taking the derivative of (\ref{eq:8ab}b) with respect to $v$ and making use of (\ref{eq:8ab}a), (\ref{eq:9ab}a) we obtain the following equation along $C^-$
\begin{align}
  \label{eq:65}
  \frac{d\delta}{du}=f_4\frac{d\nu}{du}+f_5\nu+f_6\delta.
\end{align}
Taking the derivative of (\ref{eq:9ab}a) with respect to $u$ and the derivative of (\ref{eq:9ab}b) with respect to $v$, subtracting the resulting equations from each other, we arrive at
\begin{align}
  \label{eq:64}
  \frac{\partial^2t}{\partial u\partial v}+\frac{1}{c_+-c_-}\left(\pp{c_+}{u}\nu-\pp{c_-}{v}\mu\right)=0,
\end{align}
where we omitted the arguments of $c_\pm$. \eqref{eq:64} becomes, along $C^+$, $C^-$,
\begin{align}
  \label{eq:66}
  \frac{d\mu}{dv}&=f_7\gamma+f_8\mu,\\
  \label{eq:67}
  \frac{d\nu}{du}&=f_9\delta+f_{10}\nu,
\end{align}
respectively. \eqref{eq:63}, \eqref{eq:66} constitute a system of linear equations for $\mu$, $\gamma$ which we supplement with the initial conditions $\mu(0)=1,\gamma(0)=(d\alpha^-/du)(0)$. We deduce that $\mu$, $\gamma$ do not blow up on $\mathbb{R}^+\cup\{0\}$. \eqref{eq:65}, \eqref{eq:67} constitute a system of linear equations for $\nu$, $\delta$ which we complement with the initial conditions $\nu(0)=1,\delta(0)=(d\beta^+/dv)(0)$. We deduce that $\nu$, $\delta$ do not blow up for $u\in[0,\overline{u})$. In view of \eqref{eq:9} also $\partial r/\partial u$, $\partial r/\partial v$ do not blow up along $C^+$ and $C^-$. Therefore, the derivatives of first order of $\alpha$, $\beta$, $t$ and $r$ do not blow up on $C^+$, $C^-$.

Continuing in a similar manner we see that also all derivatives of higher order satisfy linear equations along $C^+$ and $C^-$. We conclude that there is also no blowup in those quantities. We therefore have the following result which establishes characteristic initial data.

\begin{proposition}
Let $r_0>\varepsilon>0$ and let us be given free data $\beta_0^+,\alpha_0^-,t_0^+,t_0^-\in C^\infty(\mathbb{R}^+\cup\{0\})$ such that $t_+(0)=t^-(0)$. Then there exist unique smooth functions $\alpha_0^+$, $\alpha^\pm_i$, $\beta^-_0$, $\beta^\pm_i$, $t^\pm_i$, $r^\pm_0$, $r^\pm_i$, $i\geq 1$, such that $r^\pm_0(0)=r_0$ and, for $f\in\{\alpha,\beta,t,r\}$, the functions
\begin{subequations}
\begin{align}
  \label{eq:69}
  \frac{\partial^{i+j}f}{\partial u^i\partial v^j}(0,v)&=\frac{df^+_i}{dv^j}(v):i,j\geq 0,\\
  \label{eq:72}
  \frac{\partial^{i+j}f}{\partial u^j\partial v^i}(u,0)&=\frac{df^-_i}{du^j}(u):i,j\geq 0,
\end{align}
\end{subequations}
satisfy the characteristic system and derivatives of it along $C^+$, $C^-$ and the initial conditions
\refstepcounter{equation}\label{eq:71ab}
\begin{align}
  \label{eq:71}
  f^+_i(0)=\frac{d^if^-_0}{du^i}(0),\qquad f^-_i(0)=\frac{d^if^+_0}{dv^i}(0):i\geq 0.\tag{\theequation a,b}
\end{align}
$C^+$ corresponds to $\{(0,v):v\in\mathbb{R}^+\cup\{0\}\}$ and $C^-$ corresponds to $\{(u,0):u\in[0,\overline{u})\}$, where
\begin{align}
  \label{eq:73}
  \overline{u}=\sup \Big\{u'\in\mathbb{R}^+\cup\{0\}: \forall u''\in[0,u'] : r^-_0(u'')>\varepsilon\Big\}.
\end{align}
\end{proposition}
\begin{remark}
  It would suffice to give $\alpha^-_0$, $t^-_0$ on $[0,\overline{u})$.
\end{remark}

\section{Characteristic Initial Value Problem}
In the following we establish local existence of a solution to the characteristic initial value problem. We assume that, after giving free data, the constraint equations have been solved and in the following we are going to make use of this solution without further reference. We focus on establishing the solution in a region adjacent to $C^+$. Establishing the solution in a region adjacent to $C^-$ is analogous. Let us be given data along $C^+$ up to $v=v^\ast$ and along $C^-$ up to $u=u^\ast$.

\subsection{Solution in the Corner}
We define
\begin{align}
  \label{eq:11}
  I_a&\defeq \left\{(0,v)\in\mathbb{R}^2:0\leq v\leq a\right\},\\
  \Pi_{ab}&\defeq\left\{(u,v)\in\mathbb{R}^2:0\leq u\leq a,0\leq v\leq b\right\},
\end{align}
(see figure \ref{domainpiab}).

\begin{figure}[h!]
\begin{center}
\begin{tikzpicture}
\filldraw [gray!30] (0,0) -- (-0.6,0.6) --(0.4,1.6) -- (1,1);
\draw [->](0,0) -- (-2,2);
\draw [->](0,0) -- (2,2);
\node at (-2.2,2.2) {$u$};
\node at (2.2,2.2) {$v$};
%\draw (-0.2,0.2) -- (1.5,1.9) node[midway,sloped,above] {$u=h$};
\draw (0.4,1.6) -- (1,1) ;
\draw (-0.6,0.6) --(0.4,1.6);
\draw [dashed] (0.4,1.6) -- (-1,3) node[midway,sloped,above] {$v=a$};
\draw [dashed] (0.4,1.6) -- (1.8,3) node[midway,sloped,above] {$u=b$};
\node at (-1.4,0.8) {$C^-$};
\node at (1.4,0.8) {$C^+$};
\node at (0.25,0.8) {$\Pi_{ab}$};
\end{tikzpicture}
\end{center}
\caption{The domain $\Pi_{ab}$}
\label{domainpiab}
\end{figure}
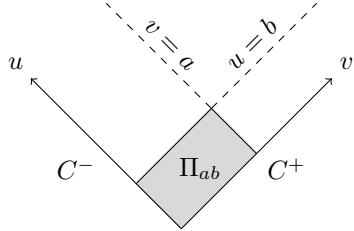

Recall that $\nu(0,v)=1$. Let
\begin{align}
  \label{eq:26}
  a_0\defeq \max_{I_{v^\ast}}|\alpha|,\qquad b_0\defeq \max_{I_{v^\ast}}|\beta|,\qquad d_0\defeq \max_{I_{v^\ast}}|\delta|
\end{align}
and let
\begin{align}
  \label{eq:14}
r_m\defeq \min_{I_{v^\ast}}r,\qquad r_M\defeq \max_{I_{v^\ast}}r.
\end{align}
Let $l>1$ and let
\begin{align}
  \label{eq:82}
  A\defeq la_0,\qquad B\defeq lb_0,\qquad D\defeq ld_0.
\end{align}
Let us choose $0<h<u^\ast$ such that for $u\in[0,h]$ we have
\begin{align}
  \label{eq:151}
  |\alpha(u,0)|<A,\qquad |\beta(u,0)|<B,\qquad |\delta(u,0)|<D,
\end{align}
\begin{align}
  \label{eq:152}
  \tfrac{1}{2}r_m<r(u,0)<\tfrac{3}{2}r_M.
\end{align}
Recall that $\mu(u,0)=1$. In the following we generalize our discussion to the case $\mu(u,0)\neq 1$, the case $\mu(u,0)=1$ being a trivial subcase. Let
\begin{align}
  \label{eq:76}
  m_0\defeq \sup_{u\in[0,h]}|\mu(u,0)|,\qquad g_0\defeq\sup_{u\in[0,h]}|\gamma(u,0)|.
\end{align}
Let
\begin{align}
  \label{eq:50}
  R_l&\defeq \left\{(\alpha,\beta)\in\mathbb{R}^2:|\alpha|\leq A,|\beta|\leq B\right\},\\
  \label{eq:3}
\Omega_l&\defeq R_l\times \left[\tfrac{1}{2}r_m,\tfrac{3}{2}r_M\right].
\end{align}

In the following we establishes bounds on $|\gamma|$, $|\mu|$ once bounds for all other quantities have been established.

Making use of the characteristic system we obtain
\begin{subequations}
\begin{align}
  \label{eq:150}
  \pp{\gamma}{v}&=A_1\nu\gamma+B_1\mu\delta+C_1\mu\nu,\\
  \pp{\mu}{v}&=A_2\nu\gamma+B_2\mu\delta+C_2\mu\nu,
\end{align}
\end{subequations}
where
\begin{align}
  \label{eq:153}
  A_1&=\frac{1}{r}\left\{(\alpha-\beta)(\tfrac{1}{4}-\tfrac{1}{2}\eta')-\eta\right\},\\
  B_1&=\frac{\alpha-\beta}{2r}(\tfrac{1}{2}+\eta'),\\
  C_1&=\frac{\eta(\alpha-\beta)}{r^2}(\alpha-\beta-2\eta),\\
  A_2=B_2&=-\frac{1}{2\eta}(\tfrac{1}{2}+\eta'),\\
  C_2&=\frac{\alpha-\beta}{r}(\eta'-\tfrac{1}{2})
\end{align}
and we denote $\eta'\defeq d\eta/d\chi^\dagger$ (see \eqref{eq:23}). Defining
\begin{align}
  \label{eq:154}
  a_1&\defeq A_1\nu,\qquad b_1\defeq B_1\delta+C_1\nu,\\
  b_2&\defeq A_2\nu,\qquad a_2\defeq B_2\delta+C_2\nu,
\end{align}
we arrive at
\begin{align}
  \label{eq:155}
  \pp{\gamma}{v}&=a_1\gamma+b_1\mu,\\
  \label{eq:157}
  \pp{\mu}{v}&=a_2\mu+b_2\gamma.
\end{align}
Now we define
\begin{align}
  \label{eq:156}
  Q_i\defeq \sup_{\stackrel{(\alpha,\beta,r)\in\Omega_l}{|\nu|\leq l}}|a_i|,\qquad S_i\defeq \sup_{\stackrel{(\alpha,\beta,r)\in\Omega_l}{|\nu|\leq l,|\delta|\leq D}}|b_i|,\qquad i=1,2.
\end{align}
Then, the system \eqref{eq:155}, \eqref{eq:157} implies
\begin{align}
  \label{eq:158}
  |\gamma(u,v)|&\leq e^{vQ_1}\left\{|\gamma(u,0)|+S_1\int_0^v|\mu|(u,v')dv'\right\},\\
|\mu(u,v)|&\leq e^{vQ_2}\left\{|\mu(u,0)|+S_2\int_0^v|\gamma|(u,v')dv'\right\}.
\end{align}
Substituting one of these equations into the other and making use of Gronwall's inequality we obtain
\begin{align}
  \label{eq:159}
  |\gamma(u,v)|&\leq f_1(v)|\gamma(u,0)|+f_2(v)|\mu(u,0)|,\\
  |\mu(u,v)|&\leq f_3(v)|\mu(u,0)|+f_4(v)|\gamma(u,0)|,
\end{align}
where
\begin{align}
  \label{eq:162}
  f_1(v)&=e^{vQ_1}\left\{1+vS_1S_2e^{v(Q_1+Q_2)}\int_0^ve^{-S_1S_2\int_{v'}^vv''e^{v''(Q_1+Q_2)}dv''}dv'\right\},\\
  f_2(v)&=vS_1e^{vQ_2}f_1(v)
\end{align}
and analogous expressions hold for $f_3$, $f_4$. We define
\begin{align}
  \label{eq:192}
  \overline{f}_1(v)&\defeq e^{vQ_1}\left\{1+v^2S_1S_2e^{v(Q_1+Q_2)}\right\},\\
  \overline{f}_2(v)&\defeq vS_1e^{vQ_2}\overline{f}_1(v)
\end{align}
and analogously we define $\overline{f}_3$, $\overline{f}_4$. We have
\begin{align}
  \label{eq:193}
  f_i\leq \overline{f}_i:i=1,\ldots,4.
\end{align}
We note that $\overline{f}_i:i=1,\ldots,4$ are strictly increasing functions satisfying
\begin{align}
  \label{eq:160}
  \lim_{v\rightarrow 0}\overline{f}_i=1:i=1,3,\qquad \lim_{v\rightarrow 0}\overline{f}_k=0:k=2,4.
\end{align}
We make the definitions
\begin{align}
  \label{eq:161}
  G\defeq \overline{f}_1(v^\ast)g_0+\overline{f}_2(v^\ast)m_0,\\
    \label{eq:4}
  M\defeq \overline{f}_3(v^\ast)g_0+\overline{f}_4(v^\ast)m_0.
\end{align}
We have proved the following:
\begin{lemma}\label{lemma_gamma_mu}
Let us be given a solution of the characteristic system in $\Pi_{uv}$ for $u\in(0,h]$, $v\in(0,v^\ast]$ such that this solution satisfies the bounds
\begin{align}
  \label{eq:53}
  |\nu|\leq l,\qquad |\alpha|\leq A,\qquad |\beta|\leq B,\qquad |\delta|\leq D,\qquad \tfrac{1}{2}r_m\leq r\leq \tfrac{3}{2}r_M.
\end{align}
Then the constants $G$, $M$ given by \eqref{eq:161}, \eqref{eq:4} respectively, satisfy $G> g_0$, $M> m_0$ and we have
\begin{align}
  \label{eq:56}
  |\gamma|\leq G,\qquad |\mu|\leq M,\qquad
\end{align}
with strict inequalities for $v<v^\ast$. The constants $G$, $M$ depend on $A$, $B$, $D$, $l$, $r_m$, $r_M$, $g_0$, $m_0$.
\end{lemma}

From the characteristic system we have
\begin{align}
  \label{eq:167}
  \alpha(u,v)&=\alpha(u,0)+\int_0^v(\nu F)(u,v')dv',\\
  \label{eq:165}
  \beta(u,v)&=\beta(0,v)+\int_0^u(\mu F)(u',v)du'.
\end{align}
Introducing the two functions
\begin{align}
  \label{eq:168}
  K\defeq \frac{1}{c_+-c_-}\pp{c_+}{u},\qquad L\defeq \frac{1}{c_+-c_-}\pp{c_-}{v},
\end{align}
with $c_\pm=c_\pm(\alpha,\beta)$ and 
\begin{align}
  \label{eq:172}
  \pp{c_+}{u}=\pp{c_+}{\alpha}(\alpha,\beta)\pp{\alpha}{u}+\pp{c_+}{\beta}(\alpha,\beta)\pp{\beta}{u}
\end{align}
and analogous for $\partial c_-/\partial v$, we can rewrite \eqref{eq:64} as
\begin{align}
  \label{eq:169}
  \pppp{t}{u}{v}+K\nu-L\mu=0.
\end{align}

We now construct a solution of the characteristic system as the limit of a sequence of functions $((\alpha_n,\beta_n,t_n,r_n);n=0,1,2,\ldots)$ defined on $\Pi_{h\varepsilon}$. The sequence is generated by the following iteration. We first define $(\alpha_0,\beta_0)$ setting
\begin{align}
  \label{eq:170}
  \alpha_0(u,v)=\alpha(u,0),\qquad \beta_0(u,v)=\beta(0,v),
\end{align}
the right hand sides being given by the initial data on $C^-$, $C^+$ respectively. We then define $t_0(u,v)$ to be the solution of (see \eqref{eq:169})
\begin{align}
  \label{eq:60}
  \frac{\partial^2t_0}{\partial u\partial v}+K_0\nu_0-L_0\mu_0=0,
\end{align}
together with the initial data $t(u,0)$ and $t(0,v)=v$. We then define (see (\ref{eq:9ab}b))
\begin{align}
  \label{eq:68}
  r_0(u,v)=r(0,v)+\int_0^u(\mu_0c_{-0})(u',v)du'.
\end{align}
Then, given the iterate $(\alpha_n,\beta_n)$ we define the next iterate $(\alpha_{n+1},\beta_{n+1})$ according to the following. We define $t_n$ to be the solution of \eqref{eq:60} with $n$ in the role of $0$. Then we define $r_n$ to be the solution of \eqref{eq:68} with $n$ in the role of $0$. Then we find the next iterate $(\alpha_{n+1},\beta_{n+1})$ according to (see \eqref{eq:167}, \eqref{eq:165})
\begin{align}
  \label{eq:70}
  \alpha_{n+1}(u,v)&=\alpha(u,0)+\int_0^v(\nu_nF_n)(u,v')dv',\\
  \label{eq:102}
  \beta_{n+1}(u,v)&=\beta(0,v)+\int_0^u(\mu_nF_n)(u',v)du',
\end{align}
where
\begin{align}
  \label{eq:77}
  F_n=-\frac{\eta(\alpha_n,\beta_n)}{r_n}(\alpha_n-\beta_n).
\end{align}

In the following we are going to work with the differences
\begin{align}
  \label{eq:75}
  \alpha_n'(u,v)&\defeq \alpha_n(u,v)-\alpha(u,0),\\
  \beta_n'(u,v)&\defeq \beta_n(u,v)-\beta(0,v).
\end{align}
We note that $\alpha'_n(u,0)=\beta'_n(0,v)=0$. Let (for the definition of $\Omega_l$ see \eqref{eq:3})
\begin{align}
  \label{eq:74}
  \overline{F}\defeq \sup_{\Omega_l}|F|.
\end{align}

\begin{lemma}\label{lemma_ind}
  Consider the closed set $\mathcal{C}$ in the space $C^1(\Pi_{h\varepsilon},\mathbb{R}^2)$ of continuously differentiable maps $(u,v)\mapsto (\alpha',\beta')(u,v)$ of $\Pi_{h\varepsilon}$ into $\mathbb{R}^2$ defined by the conditions
  \begin{align}
    \label{eq:51}
    \alpha'(u,0)=\beta'(0,v)=0
  \end{align}
and the inequalities
  \begin{align}
    \label{eq:21}
    \left|\frac{\partial\alpha'}{\partial u}\right|\leq G-g_0,\qquad    \left|\frac{\partial\alpha'}{\partial v}\right|\leq l\overline{F},\qquad \left|\frac{\partial\beta'}{\partial u}\right|\leq M\overline{F},\qquad \left|\frac{\partial\beta'}{\partial v}\right|\leq D-d_0.
  \end{align}
If $h$ and $\varepsilon$ are sufficiently small, then the sequence $((\alpha'_n,\beta'_n);n=0,1,2,\ldots)$ is contained in $\mathcal{C}$.
\end{lemma}

The following proof is given in some detail in order to arrive at the explicit smallness conditions on $h$ which will then be used again in the bootstrap argument where we continue the solution (see below).
\begin{proof}
Let $(\alpha'_n,\beta'_n)\in\mathcal{C}$. We have
\begin{align}
  \label{eq:81}
  |\alpha_n(u,v)|&\leq|\alpha(u,0)|+\int_0^v\left|\pp{\alpha_n'}{v}\right|(u,v')dv'\notag\\
&\leq \sup_{u\in[0,h]}|\alpha(u,0)|+vl\overline{F},\\
  \label{eq:188}
  |\beta_n(u,v)|&\leq|\beta(0,v)|+\int_0^u\left|\pp{\beta_n'}{u}\right|(u',v)du'\notag\\
&\leq b_0+uM\overline{F}.
\end{align}
On the other hand
\begin{align}
  \label{eq:115}
  |\alpha_n(u,v)|&\leq |\alpha(0,v)|+\int_0^u\left|\pp{\alpha_n'}{u}(u',v)+\pp{\alpha}{u}(u',0)\right|du'\notag\\
&\leq a_0+uG,\\
  |\beta_n(u,v)|&\leq|\beta(u,0)|+\int_0^v\left|\pp{\beta'_n}{v}(u,v')+\pp{\beta}{v}(0,v')\right|dv'\notag\\
&\leq \sup_{u\in[0,h]}|\beta(u,0)|+vD.
\end{align}
Therefore, if we choose $\varepsilon$ and $h$ sufficiently small such that
\begin{align}
  \label{eq:83}
  \varepsilon&\leq\min\left\{\frac{B-\sup_{u\in[0,h]}|\beta(u,0)|}{D},\frac{A-\sup_{u\in[0,h]}|\alpha(u,0)|}{l\overline{F}}\right\},\\
  \label{eq:194}
  h&\leq(l-1)\min\left\{\frac{a_0}{G},\frac{b_0}{M\overline{F}}\right\},
\end{align}
we have that $(\alpha_n,\beta_n)\in R_l$.

For the following discussion of $\mu$ and $\nu$ we omit the index $n$ since it would be the only index appearing. We consider \eqref{eq:60} with $n$ in the role of $0$. Integrating with respect to $v$ and $u$ we obtain (recall that $\nu(0,v)=1$)
\begin{align}
  \label{eq:84}
  \mu(u,v)&=\mu(u,0)e^{\int_0^vL(u,v')dv'}-\int_0^ve^{\int_{v'}^vL(u,v'')dv''}\left(K\nu\right)(u,v')dv',\\
  \label{eq:121}
  \nu(u,v)&=e^{-\int_0^uK(u',v)du'}+\int_0^ue^{-\int_{u'}^uK(u'',v)du''}\left(L\mu\right)(u',v)du'.
\end{align}
We define
\begin{align}
  \label{eq:85}
  C_{\pm \alpha}\defeq \sup_{R_l}\left|\pp{c_\pm}{\alpha}\right|,\qquad C_{\pm \beta}\defeq \sup_{R_l}\left|\pp{c_\pm}{\beta}\right|,\qquad C_{+-}\defeq\sup_{R_l}\frac{1}{c_+-c_-}.
\end{align}
We have
\begin{align}
  \label{eq:86}
|L|&=\frac{1}{c_+-c_-}\left|\pp{c_-}{v}\right|\notag\\
&\leq C_{+-}\left(C_{-\alpha}l\overline{F}+C_{-\beta}D\right)\defeq \overline{L},\\
  \label{eq:109}
  |K|&=\frac{1}{c_+-c_-}\left|\pp{c_+}{u}\right|\notag\\
&\leq C_{+-}\left(C_{+\alpha}G+C_{+\beta}M\overline{F}\right)\defeq \overline{K}.
\end{align}
Therefore,
\begin{align}
  \label{eq:87}
  |\mu(u,v)|&\leq e^{v\overline{L}}\left(m_0+\overline{K}\int_0^v|\nu|(u,v')dv'\right),\\
  \label{eq:88}
  |\nu(u,v)|&\leq e^{u\overline{K}}\left(1+\overline{L}\int_0^u|\mu|(u',v)du'\right).
\end{align}
Substituting \eqref{eq:87} into \eqref{eq:88} we obtain
\begin{align}
  \label{eq:89}
  T(u,v)\leq f_1+f_2\int_0^vT(u,v')dv',
\end{align}
where
\begin{align}
  \label{eq:91}
  T(u,v)\defeq \sup_{u'\in[0,u]}|\nu(u',v)|,
\end{align}
\begin{align}
  \label{eq:90}
  f_1(u,v)\defeq e^{u\overline{K}}\left(1+u\overline{L}e^{v\overline{L}}M\right),\qquad f_2(u,v)\defeq u\overline{K}\overline{L}e^{u\overline{K}+v\overline{L}}.
\end{align}
Defining
\begin{align}
  \label{eq:92}
  \Sigma_u(v)\defeq \int_0^vT(u,v')dv',
\end{align}
\eqref{eq:89} yields
\begin{align}
  \label{eq:93}
  \frac{d\Sigma_u}{dv}(v)\leq f_1(u,v)+f_2(u,v)\Sigma_u(v),
\end{align}
which implies
\begin{align}
  \label{eq:94}
  \Sigma_u(v)\leq \int_0^vf_1(u,v')e^{\int_{v'}^vf_2(u,v'')dv''}dv'.
\end{align}
Using this in \eqref{eq:89} yields (putting back the index $n$)
\begin{align}
  \label{eq:95}
  |\nu_n(u,v)|\leq f_1(u,v)+f_2(u,v)\int_0^vf_1(u,v')e^{\int_{v'}^vf_2(u,v'')dv''}dv'\defeq F_1(u,v).
\end{align}
Substituting this into \eqref{eq:87} gives
\begin{align}
  \label{eq:96}
  |\mu_n(u,v)|\leq e^{v\overline{L}}\left\{m_0+\overline{K}\int_0^v F_1(u,v')dv'\right\}\defeq F_2(u,v).
\end{align}
We note that
\begin{align}
  \label{eq:122}
  \lim_{u\rightarrow 0}F_1=1,\qquad \lim_{v\rightarrow 0}F_2=m_0.
\end{align}
Now we choose $\varepsilon$, $h$ sufficiently small such that (recall that $m_0<M$, $1<l$)
\refstepcounter{equation}\label{eq:100ab}
\begin{align}
  \label{eq:100}
  F_2\leq M,\qquad F_1\leq l,\tag{\theequation a,b}
\end{align}
respectively, which implies
\begin{align}
  \label{eq:78}
  |\nu_n(u,v)|\leq 1,\qquad |\mu_n(u,v)|\leq M.
\end{align}

Now we look at \eqref{eq:68} with $n$ in the role of $0$ which is
\begin{align}
  \label{eq:118}
  r_n(u,v)=r(u,0)+\int_0^v(\nu_n c_{+n})(u,v')dv'.
\end{align}
This together with \eqref{eq:60} with $n$ in the role of $0$ gives
\begin{align}
  \label{eq:119}
  r_n(u,v)=r(0,v)+\int_0^u(\mu_n c_{-n})(u',v)du'.
\end{align}
Let
\begin{align}
  \label{eq:110}
  c_\pm^\dagger\defeq \sup_{R_l}|c_\pm|.
\end{align}
If we choose $\varepsilon$ and $h$ such that
\begin{align}
  \label{eq:97}
  \varepsilon&\leq \frac{1}{c_+^\dagger l}\min\left\{\inf_{u\in[0,h]}r(u,0)-\tfrac{1}{2}r_m\,\,,\,\,\tfrac{3}{2}r_M-\sup_{u\in[0,h]}r(u,0)\right\},\\
  \label{eq:195}
 h&\leq \frac{r_m}{2c_-^\dagger M},
\end{align}
we obtain
\begin{align}
  \label{eq:98}
  \tfrac{1}{2}r_m\leq r_n\leq\tfrac{3}{2}r_M,
\end{align}
hence $(\alpha_n,\beta_n,r_n)\in\Omega_l$, which implies
\begin{align}
  \label{eq:99}
  |F_n|\leq\overline{F}.
\end{align}

In view of \eqref{eq:70}, \eqref{eq:102}, we have
\begin{align}
  \label{eq:101}
  \left|\pp{\alpha'_{n+1}}{v}\right|\leq l\overline{F},\qquad \left|\pp{\beta'_{n+1}}{u}\right|\leq M\overline{F},
\end{align}
respectively. From \eqref{eq:70}, \eqref{eq:102} we have
\begin{align}
  \label{eq:103}
  \pp{\alpha'_{n+1}}{u}(u,v)&=\int_0^v\left(\nu_n\pp{F_n}{u}+\pppp{t_n}{u}{v}F_n\right)(u,v')dv',\\
  \label{eq:189}
  \pp{\beta_{n+1}'}{v}(u,v)&=\int_0^u\left(\mu_n\pp{F_n}{v}+\pppp{t_n}{u}{v}F_n\right)(u',v)du'.
\end{align}
In view of \eqref{eq:86}, \eqref{eq:109}, \eqref{eq:78} we see that
\begin{align}
  \label{eq:108}
  \left|\pppp{t}{u}{v}\right|\leq M\overline{L}+l\overline{K}.
\end{align}
Let
\begin{align}
  \label{eq:104}
  F_f&\defeq \sup_{\Omega_l}\left|\pp{F}{f}\right|:f\in\{\alpha,\beta,r\}.
\end{align}
We have
\begin{align}
  \label{eq:79}
  \left|\nu_n\pp{F_n}{u}+\pppp{t_n}{u}{v}F_n\right|&\leq l\left(F_\alpha G+F_\beta M\overline{F}+F_r c_-^\dagger M\right)+\overline{F}(M\overline{L}+l\overline{K})\defeq H_1,\\
  \label{eq:196}
  \left|\mu_n\pp{F_n}{v}+\pppp{t_n}{u}{v}F_n\right|&\leq M\left(F_\alpha l\overline{F}+F_\beta D+F_r c_+^\dagger l\right)+\overline{F}(M\overline{L}+l\overline{K})\defeq H_2.
\end{align}
Choosing $\varepsilon$, $h$ sufficiently small such that (recall that $G>g_0$)
\refstepcounter{equation}\label{eq:106ab}
\begin{align}
  \label{eq:106}
  \varepsilon\leq\frac{G-g_0}{H_1},\qquad h\leq\frac{D-d_0}{H_2},\tag{\theequation a,b}
\end{align}
respectively, we obtain
\begin{align}
  \label{eq:107}
  \left|\pp{\alpha'_{n+1}}{u}\right|\leq G-g_0,\qquad \left|\pp{\beta'_{n+1}}{v}\right|\leq D-d_0.
\end{align}
In view of \eqref{eq:101}, \eqref{eq:107} the proof is complete.
\end{proof}

\begin{comment}
\begin{remark}
  It is important to note that the smallness conditions in the above proof which are put on $h$ do not depend on the differences of quantities on $C^-$ from the supremum of these quantities on $C^-$. E.g.~the smallness condition in (\ref{eq:106ab}a) does depend on $G-g_0$.
\end{remark}
\end{comment}

\begin{lemma}
  Let the hypotheses of lemma \ref{lemma_ind} be satisfied. If $h$ and $\varepsilon$ are sufficiently small, depending on $A$, $B$, $D$, $G$, $M$, $r_m$, $r_M$, $l$ then the sequence $((\alpha_n',\beta_n');n=0,1,2,\ldots)$ is a contractive sequence in the space $C^1(\Pi_{h\varepsilon},\mathbb{R}^2)$.
\end{lemma}
\begin{proof}
We use the definition
\begin{align}
  \label{eq:80}
  \Delta_nf\defeq f_n-f_{n-1}.
\end{align}
Let
\begin{align}
  \label{eq:139}
  \Lambda\defeq \max\left\{\sup_{\Pi_{h\varepsilon}}\left|\pp{\Delta_{n}\alpha'}{u}\right|,\sup_{\Pi_{h\varepsilon}}\left|\pp{\Delta_{n}\alpha'}{v}\right|,\sup_{\Pi_{h\varepsilon}}\left|\pp{\Delta_{n}\beta'}{u}\right|+\sup_{\Pi_{h\varepsilon}}\left|\pp{\Delta_{n}\beta'}{v}\right|\right\}.
\end{align}
In the following we denote by $C$ a constant depending on $A$, $B$, $D$, $G$, $M$, $r_m$, $r_M$, $l$. From \eqref{eq:70}, \eqref{eq:102} we have
\begin{align}
  \label{eq:122}
  \left|\pp{\Delta_{n+1}\alpha'}{v}\right|&\leq C(|\Delta_n\nu|+|\Delta_n\alpha|+|\Delta_n\beta|+|\Delta_nr|),\\
  \left|\pp{\Delta_{n+1}\beta'}{u}\right|&\leq C(|\Delta_n\mu|+|\Delta_n\alpha|+|\Delta_n\beta|+|\Delta_nr|).
\end{align}
For the differences $\Delta_n\alpha$, $\Delta_n\beta$ we have
\begin{align}
  \label{eq:125}
  |\Delta_n\alpha|\leq v\sup_{\Pi_{h\varepsilon}}\left|\pp{\Delta_{n}\alpha'}{v}\right|,\qquad
|\Delta_n\beta|\leq u\sup_{\Pi_{h\varepsilon}}\left|\pp{\Delta_{n}\beta'}{u}\right|.
\end{align}
For the difference $\Delta_nr$ we use \eqref{eq:68} with $n$ and $n-1$ in the role of $0$. We obtain
\begin{align}
  \label{eq:133}
  \Delta_nr=\int_0^u(\mu_n\Delta_nc_-+c_{-,n-1}\Delta_n\mu)(u',v)du'.
\end{align}
In view of \eqref{eq:125}, we have
\begin{align}
  \label{eq:134}
  |\Delta_nc_\pm|\leq C\left(v\sup_{\Pi_{h\varepsilon}}\left|\pp{\Delta_{n}\alpha'}{v}\right|+u\sup_{\Pi_{h\varepsilon}}\left|\pp{\Delta_{n}\beta'}{u}\right|\right).
\end{align}
Now, from \eqref{eq:60} with $n$ and $n-1$ in the role of $0$ we obtain
\begin{align}
  \label{eq:135}
  \pppp{\Delta_nt}{u}{v}+K_n\Delta_n\nu-L_n\Delta_n\mu=\Xi_n,
\end{align}
where
\begin{align}
  \label{eq:136}
  \Xi_n\defeq \mu_{n-1}\Delta_nL-\nu_{n-1}\Delta_nK.
\end{align}
Making use of \eqref{eq:134} we get
\begin{align}
  \label{eq:138}
  |\Delta_nK|&\leq\left|\Delta_{n}\left(\frac{1}{c_+-c_-}\pp{c_+}{u}\right)\right|\notag\\
&\leq C\left\{v\sup_{\Pi_{h\varepsilon}}\left|\pp{\Delta_{n}\alpha'}{v}\right|+u\sup_{\Pi_{h\varepsilon}}\left|\pp{\Delta_{n}\beta'}{u}\right|+\left|\pp{\Delta_n\alpha'}{u}\right|+\left|\pp{\Delta_n\beta'}{u}\right|\right\}\notag\\
  &\leq C\Lambda,\\
  |\Delta_nL|&\leq\left|\Delta_{n}\left(\frac{1}{c_+-c_-}\pp{c_-}{v}\right)\right|\notag\\
&\leq C\left\{v\sup_{\Pi_{h\varepsilon}}\left|\pp{\Delta_{n}\alpha'}{v}\right|+u\sup_{\Pi_{h\varepsilon}}\left|\pp{\Delta_{n}\beta'}{u}\right|+\left|\pp{\Delta_n\alpha'}{v}\right|+\left|\pp{\Delta_n\beta'}{v}\right|\right\}\notag\\
  &\leq C\Lambda.
\end{align}
Therefore,
\begin{align}
  \label{eq:144}
  |\Xi_n|\leq C\Lambda.
\end{align}
Integrating \eqref{eq:135} yields
\begin{align}
  \label{eq:137}
  \Delta_n\mu(u,v)&=-\int_0^ve^{\int_{v'}^vL_n(u,v'')dv''}\left(K_n\Delta_n\nu-\Xi_n\right)(u,v')dv',\\
  \label{eq:140}
  \Delta_n\nu(u,v)&=\int_0^ue^{-\int_{u'}^uK_n(u'',v)du''}\left(L_n\Delta_n\mu+\Xi_n\right)(u',v)du'.
\end{align}
Substituting \eqref{eq:137} into \eqref{eq:140} and following a similar procedure as was carried out in the prove of the previous lemma, we arrive at
\begin{align}
  \label{eq:141}
  |\Delta_n\nu|\leq Cu\Lambda,\\
  \label{eq:148}
  |\Delta_n\mu|\leq Cv\Lambda.
\end{align}
Using this in \eqref{eq:133} we obtain
\begin{align}
  \label{eq:142}
  |\Delta_nr|\leq Cu(u+v)\Lambda.
\end{align}
Therefore,
\begin{align}
  \label{eq:143}
  \left|\pp{\Delta_{n+1}\alpha'}{v}\right|,\left|\pp{\Delta_{n+1}\beta'}{u}\right|&\leq C(u+v)\Lambda.
\end{align}

To estimate $|\partial\Delta_n\alpha'/\partial u|$ and $|\partial\Delta_n\beta'/\partial v|$ we use \eqref{eq:103}, \eqref{eq:189}. To estimate the mixed derivative on the right hand side of the resulting difference, we observe that from \eqref{eq:135}, \eqref{eq:144}, \eqref{eq:141}, \eqref{eq:148} we get
\begin{align}
  \label{eq:145}
  \left|\pppp{\Delta_nt}{u}{v}\right|\leq C\Lambda.
\end{align}
We obtain
\begin{align}
  \label{eq:146}
  \left|\pp{\Delta_{n+1}\alpha'}{u}\right|,\left|\pp{\Delta_{n+1}\beta'}{v}\right|\leq C(u+v)\Lambda.
\end{align}
In view of \eqref{eq:143}, \eqref{eq:146} and recalling that the constants in these equations depend on $A$, $B$, $D$, $G$, $M$, $r_m$, $r_M$, $l$, we see that for sufficiently small $h$ and $\varepsilon$, depending on $A$, $B$, $D$, $G$, $M$, $r_m$, $r_M$, $l$, the sequence contracts in $C^1(\Pi_{h\varepsilon},\mathbb{R}^2)$.
\end{proof}

The two lemmas above show that the sequence $(\alpha_n',\beta_n')$ converges to $(\alpha',\beta')\in \mathcal{C}$ uniformly in $\Pi_{h\varepsilon}$. Therefore we also have uniform convergence of $(\alpha_n,\beta_n)$ to $(\alpha,\beta)\in C^1(\Pi_{h\varepsilon})$. Now, \eqref{eq:141}, \eqref{eq:148} show the convergence of the derivatives of $t_n$. Therefore, the pair of integral equations \eqref{eq:84}, \eqref{eq:121} are satisfied in the limit. We denote by $t$ the limit of $(t_n)$. It then follows that the mixed derivative $\partial^2t/\partial u\partial v$ satisfies \eqref{eq:169}. In view of the Hodograph system \eqref{eq:9} the partial derivatives of $r_n$  converge uniformly in $\Pi_{h\varepsilon}$ and the limit satisfies the Hodograph system. Let us denote by $r$ the limit of $(r_n)$. We have thus found a solution of the characteristic initial value problem in $\Pi_{h\varepsilon}$. We note that the solution satisfies the bounds \eqref{eq:53}, \eqref{eq:56}.

The two previous lemmas establish the following proposition.
\begin{proposition}
Let us be given data on $C^+$ of size $a_0$, $b_0$, $d_0$, $r_m$, $r_M$ according to \eqref{eq:26}, \eqref{eq:14}. Let us be given a constant $l>1$. Let us be given data on $C^-$ which agrees with the data on $C^+$ at $(u,v)=(0,0)$ and is of size $m_0$, $g_0$ according to \eqref{eq:76}. Then, for $h$ and $\varepsilon$ sufficiently small, depending on the size of the initial data, we have existence of a $C^1$ solution of the characteristic system in $\Pi_{h\varepsilon}$ which satisfies the bounds \eqref{eq:53}, \eqref{eq:56}.
\end{proposition}

We supplement this proposition with the following uniqueness result.
\begin{proposition}
Let $(\alpha_1,\beta_1,t_1,r_1)$ and $(\alpha_2,\beta_2,t_2,r_2)$, both in $C^1(\Pi_{h\varepsilon})$, be two solutions of the characteristic system corresponding to the same initial data. Then, for $h$, $\varepsilon$ sufficiently small, depending on the size of the initial data, the two solutions coincide.
\end{proposition}
Using similar estimates as in the convergence proof above, the proof is straightforward.

\subsection{Extension to a Strip}
Now we show that the solution is actually given in $\Pi_{hv^\ast}$, i.e.~we have existence of a unique continuously differentiable solution in a region adjacent to $C^+$ which extends over the full domain of the initial data and is of thickness $h$, where $h$ depends on the size of the initial data in the way given by the above propositions.

\begin{lemma}
  The solution in $\Pi_{h\varepsilon}$ can be continued to $\Pi_{hv^\ast}$.
\end{lemma}
\begin{proof}
As a consequence of the results above, the solution satisfies the following bounds in $\Pi_{h\varepsilon}$:
\begin{align}
  \label{eq:111}
  |\nu|\leq l,\qquad |\alpha|\leq A,\qquad |\beta|\leq B,\qquad |\delta|\leq D,\qquad \tfrac{1}{2}r_m\leq r\leq \tfrac{3}{2}r_M.\tag{BA$\sharp$}
\end{align}
Let
\begin{align}
  \label{eq:147}
  v_2&\defeq\sup\Big\{v\in(0,v^\ast):\eqref{eq:111} \textrm{ holds and the solution is unique for $(u,v)\in\Pi_{hv}$}\Big\}.
\end{align}
Let us assume $v_2<v^\ast$. Since the solution is continuously differentiable it is also unique at $v=v_2$. In the following all statements hold within $\Pi_{hv_2}$. From lemma \ref{lemma_gamma_mu} we have
\begin{align}
  \label{eq:112}
  |\gamma|< G,\qquad |\mu|< M.
\end{align}
From now on we will make use of \eqref{eq:111}, \eqref{eq:112} without further notice. We look at
\begin{align}
  \label{eq:113}
  \beta(u,v)=\beta(0,v)+\int_0^u(\mu F)(u',v)du'.
\end{align}
We have
\begin{align}
  \label{eq:114}
  |\beta(u,v)|&< b_0+h M\overline{F}\notag\\
  &\leq b_0+(l-1)b_0=lb_0=B,
\end{align}
where for the second inequality we used \eqref{eq:194}. We look at
\begin{align}
  \label{eq:116}
  r(u,v)=r(0,v)+\int_0^u(c_-\mu)(u',v)du'.
\end{align}
We have
\begin{align}
  \label{eq:117}
  r(u,v)&< r_M+hc_-^\dagger M\notag\\
  &\leq r_M+\tfrac{1}{2}r_m\leq \tfrac{3}{2}r_M,
\end{align}
where for the second inequality we used \eqref{eq:195}. Similarly we find $r(u,v)>\tfrac{1}{2}r_m$. Therefore,
\begin{align}
  \label{eq:120}
  \tfrac{1}{2}r_m<r<\tfrac{3}{2}r_M.
\end{align}
We have
\begin{align}
  \label{eq:123}
  |L|&=\frac{1}{c_+-c_-}\left|\pp{c_-}{v}\right|\leq C_{+-}\left(C_{-\alpha}l\overline{F}+C_{-\beta}D\right)=\overline{L},\\
  \label{eq:124}
  |K|&=\frac{1}{c_+-c_-}\left|\pp{c_+}{u}\right|< C_{+-}\left(C_{+\alpha}G+C_{+\beta}M\overline{F}\right)=\overline{K}.
\end{align}
For the definition of $\overline{L}$, $\overline{K}$ see \eqref{eq:86}, \eqref{eq:109} respectively. We deduce that there exists a constant $\overline{K}'$ such that
\begin{align}
  \label{eq:163}
  |K|\leq \overline{K}'<\overline{K}.
\end{align}
Therefore, the inequalities \eqref{eq:87}, \eqref{eq:88} hold for $\mu$, $\nu$ with $\overline{K}'$ in the role of $\overline{K}$. This then implies that \eqref{eq:95} holds with $\overline{K}'$ in the role of $\overline{K}$, i.e.
\begin{align}
  \label{eq:130}
  |\nu(u,v)|\leq F_1'(u,v),
\end{align}
where by $F_1'(u,v)$ we denote $F_1(u,v)$ with $\overline{K}'$ in the role of $\overline{K}$. Since
\begin{align}
  \label{eq:164}
  F_1'(u,v)&<F_1(u,v)\notag\\
  &\leq F_1(h,v),
\end{align}
together with (\ref{eq:100ab}b), we obtain
\begin{align}
  \label{eq:132}
  |\nu(u,v)|<l.
\end{align}
Using \eqref{eq:123}, \eqref{eq:124} together with \eqref{eq:112} we obtain
\begin{align}
  \label{eq:126}
  \left|\pppp{t}{u}{v}(u,v)\right|< M\overline{L}+l\overline{K}.
\end{align}
We also have
\begin{align}
  \label{eq:131}
  \left|\mu\pp{F}{v}\right|<M\left(F_\alpha l\overline{F}+F_\beta D+F_r c_+^\dagger l\right).
\end{align}
Taking the derivative of \eqref{eq:165} with respect to $v$ and using the previous two estimates we obtain (for the definition of $H_2$ see \eqref{eq:196})
\begin{align}
  \label{eq:166}
  \left|\mu\pp{F}{v}+\pppp{t}{u}{v}F\right|<H_2.
\end{align}
Therefore,
\begin{align}
  \label{eq:127}
  |\delta(u,v)|&< d_0+hH_2\notag\\
  &\leq D,
\end{align}
where for the last inequality we used (\ref{eq:106ab}b). Finally we consider
\begin{align}
  \label{eq:128}
  \alpha(u,v)=\alpha(0,v)+\int_0^u\gamma(u',v)du'.
\end{align}
We have
\begin{align}
  \label{eq:129}
  |\alpha(u,v)|&< a_0+hG\notag\\
  &\leq A,
\end{align}
where we used \eqref{eq:194}. We have therefore established in $\Pi_{h\varepsilon}$:
\begin{align}
  \label{eq:149}
  |\nu|< l,\qquad |\alpha|< A,\qquad |\beta|< B,\qquad |\delta|< D,\qquad \tfrac{1}{2}r_m< r< \tfrac{3}{2}r_M\tag{BA},
\end{align}
i.e.~we have improved \eqref{eq:111}. Therefore, using the result from above, we can solve an initial value problem with corner at $(u,v)=(0,v_2)$, thus establishing a unique solution in $\Pi_{h(v_2+\varepsilon)}$, for some $\varepsilon<v^\ast-v_2$ which satisfies the bounds \eqref{eq:111}. This implies $v_2=v^\ast$.
\end{proof}

\subsection{Higher Regularity}
We now establish uniform bounds in $\Pi_{hv^\ast}$ for the partial derivatives of $\alpha$, $\beta$, $t$ and $r$ to arbitrary order.
\begin{lemma}
The partial derivatives of $\alpha$, $\beta$, $t$, $r$ to all order are, in absolute value, uniformly bounded in $\Pi_{hv^\ast}$.
\end{lemma}
\begin{proof}
We establish such bounds by induction. Let $(P_{n-1})$ be the proposition
\begin{align}
  \label{eq:105}
  \sup_{\Pi_{hv^\ast}}\left|\frac{\partial^{n-1}}{\partial u^i\partial v^j}(\alpha,\beta,t,r)\right|\leq C,\quad i+j=n-1.\tag{$P_{n-1}$}
\end{align}
We have already established proposition ($P_{1}$). Suppose that proposition ($P_k$) holds for $k=1,\ldots,n-1$. In the following we denote by $F_i$ a function in $\Pi_{hv^\ast}$ which involves only $(n-1)'\textrm{th}$ order derivatives of $\alpha$, $\beta$, $t$, $r$ and which is, therefore, uniformly bounded. Functions carrying the same index can change from line to line. We first deal with the mixed derivatives of order $n$. Let $1\leq i,j\leq n-1$, $i+j=n$. In view of \eqref{eq:169} we have
\begin{align}
  \label{eq:171}
  \frac{\partial^nt}{\partial u^i\partial v^j}=\frac{\partial^{n-2}}{\partial u^{i-1}\partial v^{j-1}}(L\mu-K\nu).
\end{align}
Because the right hand side involves only derivatives of order $n-1$ and lower, by the inductive hypothesis, it is bounded. Therefore
\begin{align}
  \label{eq:173}
  \left|\frac{\partial^nt}{\partial u^i\partial v^j}\right|\leq C.
\end{align}
In view of (\ref{eq:8ab}a) we have
\begin{align}
  \label{eq:174}
  \frac{\partial^n\alpha}{\partial u^i\partial v^j}&=\frac{\partial^{n-1}}{\partial u^i\partial v^{j-1}}\left(\pp{\alpha}{v}\right)=\frac{\partial^{n-1}}{\partial u^i\partial v^{j-1}}\left(\nu F\right)\notag\\
  &=F_1\frac{\partial^nt}{\partial u^i\partial v^j}+F_2.
\end{align}
In view of (\ref{eq:8ab}b), the same holds for $\partial^n\beta/\partial u^i\partial v^j$. Therefore, together with \eqref{eq:173}, we obtain
\begin{align}
  \label{eq:175}
  \left|\frac{\partial^n\alpha}{\partial u^i\partial v^j}\right|,\left|\frac{\partial^n\beta}{\partial u^i\partial v^j}\right|\leq C.
\end{align}
We turn to the pure derivatives of order $n$. By \eqref{eq:8} we have
\begin{align}
  \label{eq:176}
  \frac{\partial^n\alpha}{\partial v^n}&=\frac{\partial^{n-1}}{\partial v^{n-1}}(\nu F)=F_1\frac{\partial^nt}{\partial v^n}+F_2,\\
  \label{eq:179}
  \frac{\partial^n\beta}{\partial u^n}&=\frac{\partial^{n-1}}{\partial v^{n-1}}(\mu F)=F_1\frac{\partial^nt}{\partial u^n}+F_2,
\end{align}
while from \eqref{eq:167} we have
\begin{align}
  \label{eq:177}
  \frac{\partial^n\alpha}{\partial u^n}(u,v)&=\frac{\partial^n\alpha}{\partial u^n}(u,0)+\int_0^v\left(\frac{\partial^n}{\partial u^n}(\nu F)\right)(u,v')dv'\notag\\
  &=F_1(u,v)+\int_0^v\left(F_2\frac{\partial^{n+1}t}{\partial u^n\partial v}+F_3\frac{\partial^nF}{\partial u^n}+F_4\frac{\partial^nt}{\partial u^{n-1}\partial v}\right)(u,v')dv'.
\end{align}
The mixed derivative of $t$ of order $n$ is taken care of by \eqref{eq:173}. For the mixed derivative of order $n+1$ of $t$ we look at
\begin{align}
  \label{eq:178}
  \frac{\partial^{n+1}t}{\partial u^n\partial v}&=\frac{\partial^{n-1}}{\partial u^{n-1}}(L\mu-K\nu)\notag\\
&=F_1+F_2\frac{\partial^nt}{\partial u^n}+F_3\frac{\partial^n\alpha}{\partial u^n}+F_4\frac{\partial^n\beta}{\partial u^n},
\end{align}
where for the second equality we use that we already have bounds for the mixed derivatives to order $n$ of $\alpha$, $\beta$, $t$. Using also \eqref{eq:179} we arrive at
\begin{align}
  \label{eq:180}
  \frac{\partial}{\partial v}\left(\frac{\partial^nt}{\partial u^n}\right)=F_1+F_2\frac{\partial^nt}{\partial u^n}+F_3\frac{\partial^n\alpha}{\partial u^n}.
\end{align}
This implies
\begin{align}
  \label{eq:181}
  \frac{\partial^nt}{\partial u^n}(u,v)=F_1(u,v)+\int_0^v\left(F_2\frac{\partial^n\alpha}{\partial u^n}\right)(u,v')dv'.
\end{align}
Using now
\begin{align}
  \label{eq:182}
  \frac{\partial^nF}{\partial u^n}=F_1+F_2\frac{\partial^n\alpha}{\partial u^n}+F_3\frac{\partial^n\beta}{\partial u^n}+F_4\frac{\partial^nt}{\partial u^n},
\end{align}
together with \eqref{eq:178}, \eqref{eq:181} in \eqref{eq:177} we obtain
\begin{align}
  \label{eq:183}
  \frac{\partial^n\alpha}{\partial u^n}(u,v)=F_1(u,v)+\int_0^v\left(F_2\frac{\partial^n\alpha}{\partial u^n}+F_3\frac{\partial^nt}{\partial u^n}\right)(u,v')dv'
\end{align}
This together with \eqref{eq:181} yields the following system of inequalities
\begin{align}
  \label{eq:184}
  \left|\frac{\partial^n\alpha}{\partial u^n}(u,v)\right|&\leq C+C'\int_0^v\left(\left|\frac{\partial^n\alpha}{\partial u^n}\right|+\left|\frac{\partial^nt}{\partial u^n}\right|\right)(u,v')dv',\\
\left|\frac{\partial^nt}{\partial u^n}(u,v)\right|&\leq C+C'\int_0^v\left|\frac{\partial^n\alpha}{\partial u^n}\right|(u,v')dv',
\end{align}
which implies
\begin{align}
  \label{eq:185}
  \left|\frac{\partial^n\alpha}{\partial u^n}(u,v)\right|,\left|\frac{\partial^nt}{\partial u^n}(u,v)\right|\leq C.
\end{align}
Similarly we obtain
\begin{align}
  \label{eq:186}
  \left|\frac{\partial^n\beta}{\partial v^n}(u,v)\right|,\left|\frac{\partial^nt}{\partial v^n}(u,v)\right|\leq C.
\end{align}
In view of \eqref{eq:176}, \eqref{eq:179} these imply
\begin{align}
  \label{eq:187}
  \left|\frac{\partial^n\beta}{\partial u^n}(u,v)\right|,\left|\frac{\partial^n\alpha}{\partial v^n}(u,v)\right|\leq C.
\end{align}
These together with \eqref{eq:173}, \eqref{eq:175} imply proposition $(P_n)$.
\end{proof}

The above existence, uniqueness, continuation and regularity result can be carried out for a region adjacent to $C^-$ as well. Together with the solution of the constraint equations from the previous section we arrive at the following result.

\begin{theorem}
  Let $r_0>0$ and let us be given free data  $\beta^+, t^+\in C^\infty[0,v^\ast]$, $\alpha^-,t^-\in C^\infty[0,u^\ast]$ such that $t_+(0)=t^-(0)$. Then, for $h'$, $h''$ sufficiently small depending on the size of the data, there exists a unique smooth solution $\alpha$, $\beta$, $t$, $r$ of the characteristic system of equations for $(u,v)\in \Pi_{u^\ast h'}\cup\Pi_{h''v^\ast}$, where
  \begin{align}
    \label{eq:197}
    \Pi_{ab}&\defeq\left\{(u,v)\in\mathbb{R}^2:0\leq u\leq a,0\leq v\leq b\right\},
  \end{align}
such that
\begin{alignat}{3}
  \label{eq:198}
  r(0,0)=r_0,\qquad t(0,v)&=t^+(v),&\qquad t(u,0)&=t^-(u),\\
  \beta(0,v)&=\beta^-(v), & \alpha(u,0)&=\alpha^-(u).
\end{alignat}
\end{theorem}

\begin{figure}[h!]
\begin{center}
\begin{tikzpicture}
\filldraw [gray!30] (0,0) -- (-0.2,0.2) --(2.5,2.9) -- (2.7,2.7);
\filldraw [gray!30] (0,0) -- (0.3,0.7) --(-1.2,2.2) -- (-1.7,1.7);
\draw [->](0,0) -- (-2.3,2.3);
\draw [->](0,0) -- (3.3,3.3);
\node at (-2.5,2.5) {$u$};
\node at (3.5,3.5) {$v$};
\draw (0.3,0.7) -- (2.5,2.9) node[midway,sloped,above] {$u=h''$};
\draw (0.3,0.7) -- (-1.2,2.2) node[midway,sloped,above] {$v=h'$};
\draw (2.5,2.9) -- (2.7,2.7) ;
\draw (-1.2,2.2) -- (-1.7,1.7) ;
\draw [dashed] (2.5,2.9) -- (1.5,3.9) node[midway,sloped,above] {$v=v^\ast$};
\draw [dashed] (-1.2,2.2) -- (-0.1,3.3) node[midway,sloped,above] {$u=u^\ast$};
%\draw [dashed] (0,0.4) -- (-1,1.4) node[midway,sloped,above] {$v=\varepsilon$};
\node at (-1.4,0.8) {$C^-$};
\node at (1.4,0.8) {$C^+$};
\end{tikzpicture}
\end{center}
\caption[Maximal development, comparison]{The domain $\Pi_{u^\ast h'}\cup\Pi_{h'' v^\ast}$}
\label{position}
\end{figure}
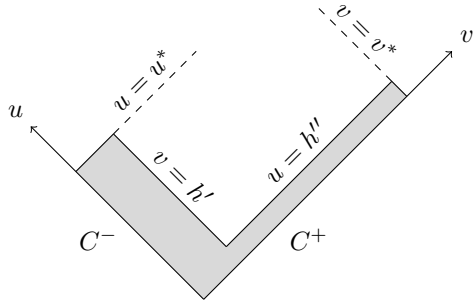

\subsection{Solution in the $t$-$r$-plane\label{tr_solution}}
We consider the map $(u,v)\mapsto(t,r)$. In view of
\begin{align}
  \label{eq:190}
\left|\begin{array}{cc}
      \partial t/\partial u & (\partial t/\partial u)c_- \\
      \partial t/\partial v & (\partial t/\partial v)c_+
    \end{array}\right|&=\left|\begin{array}{cc}
  \mu & \mu c_- \\
  \nu & \nu c_+
\end{array}\right|\notag\\
&=2\mu\nu \eta,
\end{align}
a solution of the characteristic system corresponds to a solution in the $t$-$r$-plane as long as
\begin{align}
  \label{eq:191}
  \mu,\nu>0.
\end{align}
We recall that $\mu(u,0)=\nu(0,v)=1$. Let us assume that $\mu(0,v),\nu(u,0)>0$. Since we have bounds on the second derivatives of $t$ we deduce that for $h'$, $h''$ sufficiently small, the solution in $\Pi_{u^\ast h'}\cup\Pi_{h''v^\ast}$ corresponds to a solution in the $t$-$r$-plane.


\begin{bibdiv}
\begin{biblist}

\bib{Riemann}{article}{
title={\"Uber die Fortpflanzung ebener Luftwellen von endlicher Schwingungsweite},
author={Riemann, Bernhard},
date={1860},
publisher={Hermann},
journal = {Abhandlungen der K\"{o}niglichen Gesellschaft der Wissenschaften zu G\"{o}ttingen}
}

\bib{Courant_Friedrichs}{book}{
title = {Supersonic Flow and Shock Waves},
author = {Courant, R.},
author = {Friedrichs, K.O.},
year = {1948}
publisher = {Springer-Verlag},
series={Applied Mathematical Sciences},
year={1999},
publisher={Springer New York}
}

\begin{comment}
\bib{Luk}{article}{
title={ On the local existence for the characteristic initial value problem in general relativity,}
author={Luk, J.}
year={2012}
journal={Int. Mat. Res. Notices}
}

\bib{Rendall}{article}{
title={Reduction of the Characteristic Initial Value Problem to the Cauchy Problem and its Applications to the Einstein Equations,}
author={Rendall, A.~D.}
year={1990}
journal={Proc.~R.~Soc.~Lond.~A}
}

\bib{Christodoulou}{book}{
title = {The Formation of Black Holes in General Relativity},
author = {Christodoulou, D.},
year = {2009}
series={EMS Monographs in Mathematics},
year={2009},
publisher={Z\"urich: European Mathematical Society (EMS)}
}
\end{comment}

\bib{Christodoulou_2007}{book}{
title = {The Formation of Shocks in 3-Dimensional Fluids},
author = {Christodoulou, D.},
year = {2007}
series={EMS Monographs in Mathematics},
year={2007},
publisher={Z\"urich: European Mathematical Society (EMS)}
}

\bib{Christodoulou_Lisibach}{article}{
title = {Shock Development in Spherical Symmetry, Arxiv:},
author = {Christodoulou, D.},
author = {Lisibach, A.},
year = {2015}
NOTE = {preprint, \url{http://arxiv.org/abs/1501.04235}}
}

\end{biblist}
\end{bibdiv}
\end{document}